\documentclass[11pt]{amsart}
\usepackage{geometry}                % See geometry.pdf to learn the layout options. There are lots.               
\usepackage{amsmath}
\usepackage{amscd}
\usepackage{amsthm}
\usepackage{graphicx}
\usepackage{amssymb}
\usepackage{epstopdf}
\theoremstyle{plain}
\newtheorem{Thm}{Theorem}[section]
\newtheorem{Conj}[Thm]{Conjecture}

\newtheorem{Cor}[Thm]{Corollary}
\newtheorem{Lem}[Thm]{Lemma}

\theoremstyle{definition}
\newtheorem{Defn}[Thm]{Definition}
\newtheorem{Expl}[Thm]{Example}

\theoremstyle{remark}
\newtheorem{Rem}[Thm]{Remark}
\newtheorem{Que}[Thm]{Question}

\numberwithin{equation}{section}

\title{Birational geometry and derived categories}
\author{Yujiro Kawamata}
\begin{document}
\maketitle
%\section{}
%\subsection{}

\begin{abstract}
This paper is based on a talk at a conference \lq\lq JDG 2017: Conference on Geometry and Topology''.
We survey recent progress on the DK hypothesis connecting the birational geometry and 
the derived categories stating that the $K$-equivalence of smooth projective varieties should correspond to the equivalence of 
their derived categories, and the $K$-inequality to the fully faithful embedding.  
We consider two kinds of factorizations of birational maps between algebraic varieties into elementary ones; 
those into flips, flops and divisorial contractions according to the minimal model program, 
and more traditional weak factorizations into blow-ups and blow-downs with smooth centers.
We review major approaches towards the DK hypothesis for flops between smooth varieties.
The latter factorization leads to an weak evidence of the DK hypothesis at the Grothendieck 
ring level.
DK hypothesis is proved in the case of toric or toroidal maps, and leads to the derived McKay
correspondence for certain finite subgroups of $GL(n,\mathbf{C})$.
\end{abstract}

\tableofcontents

%14 E 16, 14 E 30

%%%%%%%%%%%%%%%%%%%%%
%%%%%%%%%%%%%%%%%%%%%
%%%%%%%%%%%%%%%%%%%%%
\section{Introduction: DK hypothesis}\label{intro}

We work over $\mathbf{C}$ or an algebraically closed field of characteristic $0$.

This is a continuation of \cite{DK} and \cite{Seattle}.
In \cite{DK} we studied the parallelism between $D$ (derived) and $K$ (canonical) equivalences
inspired by earlier results in \cite{Bondal}, \cite{Orlov}, \cite{BO}, \cite{BKR}, \cite{Bridgeland2}.
Namely we asked the following question: $K$-equivalence implies $D$-equivalence?
The converse direction from $D$ to $K$ is partly proved in some cases thanks to Orlov's representation theorem \cite{Orlov2}.
We asked furthermore a generalized question: $K$-inequality implies $D$-embedding?
We note that a fully faithful embedding between derived categories of smooth projective varieties 
implies the semi-orthogonal decomposition (SOD) by \cite{Bondal-VdBergh}.
These questions are positively answered in the case of toric and toroidal morphisms \cite{log-crep}, 
\cite{toric}, \cite{toricII}, \cite{toricIII}.
We also consider the implication of $K$-equivalence at the level of Grothendieck ring of varieties and 
categories, thereby giving a weak evidence for our questions.

The canonical divisor $K_X$ is the most fundamental invariant of an algebraic variety $X$.
We chase the change of $K$ or its log version $K_X+B$ when we run a minimal model program (MMP). 
$K$ is also related to the study of derived categories because 
$K$, or more precisely the tensoring of $\omega_X[\dim X]$, appears as the Serre functor 
of a derived category of a smooth projective variety $X$.
We note that the Serre functor is determined only by the categorical data.

We note that an algebraic variety can be reconstructed from an abelian category of coherent sheaves.
Thus a nontrivial equivalence or a fully faithful embedding between the derived categories of 
birationally equivalent algebraic varieties occurs only at the level of derived categories. 
A derived category may have different $t$-structures corresponding to different varieties.
It is interesting to investigate the wall crossing phenomena corresponding to the birational change of varieties.

\begin{Defn}
Two smooth projective varieties $X$ and $Y$ are said to be {\em $K$-equivalent}, and denoted by $X \sim_K Y$, if there
is a third smooth projective variety $Z$ with birational morphisms $f: Z \to X$ and $g: Z \to Y$ such that 
the pull-backs of canonical divisors are linear equivalent: $f^*K_X \sim g^*K_Y$.
Similarly, a {\em $K$-inequality} $X \le_K Y$ is defined by $f^*K_X + E \sim g^*K_Y$ 
for some effective divisor $E$ on $Z$.
\end{Defn}

We can replace $Z$ by any other higher model in the above definition, i.e., by $Z'$ with a birational morphism $Z' \to Z$.
Thus the above properties depend only on the choice of a birational map $g \circ f^{-1}: X \dashrightarrow Y$.

We denote by $D^b(\text{coh}(X))$ the bounded derived category of coherent sheaves on $X$.
Our main question is the following:

\begin{Conj}[{\em DK-hypothesis}]
Let $X$ and $Y$ be smooth projective varieties.
If $X \sim_K Y$, then there is an equivalence of 
triangulated categories $D^b(\text{coh}(X)) \cong D^b(\text{coh}(Y))$.
If $X \le_K Y$, then there is a fully faithful functor of triangulated categories $D^b(\text{coh}(X)) \to D^b(\text{coh}(Y))$.
\end{Conj}

\begin{Defn}
A {\em semi-orthogonal decomposition} (SOD) of a triangulated category $\mathcal{A}$, denoted by 
\[
\mathcal{A} = \langle \mathcal{C}, \mathcal{B} \rangle
\]
is defined by the following conditions:

\begin{enumerate}
\item $\mathcal{B}$ and $\mathcal{C}$ are triangulated subcategories of $\mathcal{A}$ which are 
orthogonal in one direction in the sense that 
$\text{Hom}(b,c) = 0$ if $b \in \mathcal{B}$ and $c \in \mathcal{C}$.

\item $\mathcal{B}$ and $\mathcal{C}$ generate $\mathcal{A}$ in the sense that, for arbitrary $a \in \mathcal{A}$, 
there exist $b \in \mathcal{B}$ and $c \in \mathcal{C}$ such that there is a distinguished triangle
\[
b \to a \to c \to b[1] 
\] 
\end{enumerate}
\end{Defn}

By Bondal and Van den Bergh \cite{Bondal-VdBergh}, 
if there is a fully faithful functor $\Phi: D^b(\text{coh}(X)) \to D^b(\text{coh}(Y))$ for
smooth projective varieties $X$ and $Y$, then there is a semi-orthogonal decomposition
\[
D^b(\text{coh}(Y)) = \langle \mathcal{C}, \Phi(D^b(\text{coh}(X))) \rangle
\]
where $\mathcal{C}$ is the {\em right orthogonal complement} of the image of $\Phi$ defined by
\[
\mathcal{C} = \{c \in \mathcal{A} \mid \text{Hom}(\Phi(b),c) = 0 \,\,\, \forall b \in D^b(\text{coh}(X)) \}. 
\]

\begin{Expl}
Let $f: Y \to X$ be a birational {\em morphism} between smooth projective varieties.
Then $K_Y = f^*K_X + E$, where $E$ is an effective divisor which is locally expressed as the divisor of the 
Jacobian determinant.
Correspondingly, the pull back functor $Lf^*: D^b(\text{coh}(X)) \to D^b(\text{coh}(Y))$ is fully faithful.
The right orthogonal complement is given by
$\mathcal{C} = \{c \in D^b(\text{coh}(Y)) \mid Rf_*c = 0\}$.
\end{Expl}

An evidence of the conjecture is given by a partial converse: if there is an equivalence of the derived
categories of smooth projective varieties, then these varieties are $K$-equivalent under some additional conditions
(Theorem~\ref{converse1}).
This is a consequence of Orlov's representability theorem (\cite{Orlov2}).
We note that the full converse is not true even in the equivalence case (Remark~\ref{Uehara}).

Any birational morphism should be decomposed into elementary birational maps.
In a traditional way, elementary maps are blowing-ups and downs.
The {\em weak factorization theorem} states that any birational morphism is factorized into blow-ups and downs
(\S \ref{weak fact} Theorems~\ref{weak} and \ref{Morse}).
The derived categories changes accordingly in such a way that they inflate and deflate.
Then it is difficult to compare the end results.
But still we can say something on the Grothendieck group level of derived categories for  
$K$-equivalent varieties as shown in \S \ref{K0} (Theorem~\ref{S(X,B)}). 
This is a weak evidence for the above conjecture.

More modern minimal model program (MMP) implies that any birational morphism 
which strictly decreases the canonical divisor is factorized into flips and
divisorial contractions.
In this process, the canonical $K$ decreases in each steps, so that the comparison should be easy.
Ideally speaking, there is a parallelism between the MMP 
and the behavior of the derived category, i.e., fully faithful embeddings, as in the conjecture.
But singularities appear when we perform the MMP and cause trouble.
If the singularities are only quotient ones, then we can still treat them as if they are smooth 
using the covering stacks (\S \ref{quotient}), but we do not know how to treat more complicated singularities.
By the use of MMP, we can prove that the $K$-equivalence can be decomposed into flops in some cases
(\S \ref{factorizationMMP}).
Thus the conjecture for $K$-equivalence can be reduced to the study of flops in these cases 
up to the study of singularities.

We review several known approaches to the above 
conjecture in the case of flops between smooth varieties in \S \ref{approach}.
The first one using the moduli space of objects is the most difficult to prove, 
but reveals the deepest structure of the situation.
The second one uses the variation of the geometric invariant theory quotients (VGIT) 
and is useful when a big algebraic group acts on the whole situation.
We note that the weak factorization theorem is also an application of VGIT.
The third one using the tilting generator connects an algebraic variety to a non-commutative associative algebra 
and maybe the easiest to prove.
These methods are quite interesting in their own rights.
But it is still very far from the proof of the conjecture.

In the case of toric and toroidal morphisms, the conjecture is completely proved by using the MMP and 
the explicit method in a series of papers \cite{log-crep}, \cite{toric}, \cite{toricII}, \cite{toricIII}, which will
be reviewed in \S \ref{toric}.
The McKay correspondence can be considered as a special case of the conjecture when combined with the 
construction in \S \ref{quotient}.
In \S \ref{GL3}, we review some cases of the derived McKay correspondence.

\vskip 1pc

The author would like to thank the National Center for Theoretical Sciences at Taiwan University 
where this work was partly done. 
This work is partly supported by JSPS Grant-in-Aid (A) 16H02141.

%%%%%%%%%%%%%%%%%%%%%
%%%%%%%%%%%%%%%%%%%%%
%%%%%%%%%%%%%%%%%%%%%
%%%%%%%%%%%%%%%%%%%%%%%%%%%%%%%
\section{Partial converse}\label{converse}

The converse statement of DK hypothesis is partially true thanks to Orlov's representability theorem.
This is a supporting fact to the DK hypothesis itself.

\begin{Thm}\cite{Orlov2}\label{Orlov}
Let $X$ and $Y$ be smooth projective varieties.
Assume that there is a fully faithful functor of triangulated categories
$\Phi: D^b(\text{coh}(X)) \to D^b(\text{coh}(Y))$.
Then there exists a uniquely determined object $e \in D^b(\text{coh}(X \times Y))$ such that 
$\Phi(\bullet) \cong p_{2*}(p_1^*(\bullet) \otimes^L e)$, where $p_i$ are projections for $i=1,2$.
\end{Thm}

We note that, 
by \cite{Bondal-VdBergh}, $D^b(\text{coh}(X))$ is {\em saturated}, in the sense that $\Phi$ has always adjoints.
The same statement is generalized to the case where $X$ and $Y$ are smooth Deligne-Mumford stacks 
with projective coarse moduli spaces (\cite{equi}).
Therefore we can apply the theorem to the case of projective varieties of quotient types.

\begin{Defn}
Let $\mathcal{A}$ be a $k$-linear triangulated category.
Assume that $\text{Hom}(a,b)$ is a finite dimensional $k$-linear space for any $a,b \in \mathcal{A}$.
An auto-equivalence functor $S: \mathcal{A} \to \mathcal{A}$ is said to be a {\em Serre functor} if there
is an isomorphism $f: \text{Hom}(a,b) \cong \text{Hom}(b,S(a))^*$ 
of bifunctors $\mathcal{A}^o \times \mathcal{A} \to (k\text{-mod})$. 
\end{Defn}

It is important to note that the Serre functor is intrinsic for the category $\mathcal{A}$ if it exists, 
because $S$ is determined by $f$.
The following fact connects the birational geometry to the theory of derived categories:

\begin{Thm}
Let $X$ be a smooth projective variety.
Then the bounded derived category of coherent sheaves $D^b(\text{coh}(X))$ has 
a Serre functor defined by 
\[
S_X(\bullet) = \bullet \otimes \omega_X[\dim X]
\]
where $\omega_X = \mathcal{O}_X(K_X)$ is the {\em canonical sheaf}, i.e., $\Omega^{\dim X}_X$.
\end{Thm}

If $\tilde X$ is a smooth Deligne-Mumford stack
with projective coarse moduli space, then $D^b(\text{coh}(\tilde X))$ has also a Serre functor
(see \S \ref{quotient}).

\begin{Thm}[\cite{DK}]\label{converse1}
Let $X$ and $Y$ be smooth projective varieties.
Assume that there is an equivalence of triangulated categories
$\Phi: D^b(\text{coh}(X)) \to D^b(\text{coh}(Y))$.
Assume in addition that $\kappa(X,K_X) = \dim X$ or $\kappa(X,-K_X) = \dim X$
Then $X \sim_K Y$, i.e., 
there exists a projective variety $Z$ with birational morphisms 
$p: Z \to X$ and $q: Z \to Y$ such that $p^*K_X \sim q^*K_Y$.
\end{Thm}

\begin{proof}
Since the Serre functors $S_X$ and $S_Y$ are intrinsic for the derived categories,
they are compatible with the equivalence: $\Phi \circ S_X \cong S_Y \circ \Phi$.
If $e \in D^b(\text{coh}(X \times Y))$ is the kernel of $\Phi$, then we have 
$p_1^*\omega_X \otimes e \cong p_2^*\omega_Y \otimes e$.
It is sufficient to take $Z$ to be an irreducible component of the support of $e$ which dominates $X$ 
(\cite{DK}).
\end{proof}

\begin{Rem}\label{Uehara}
There are two rational elliptic surfaces $f: X \to \mathbf{P}^1$ and $g: Y \to \mathbf{P}^1$ such that 
$D^b(\text{coh}(X)) \cong D^b(\text{coh}(Y))$ but $X \not\sim_K Y$ (\cite{Uehara}).
In this case, we have $K_X \sim f^*\mathcal{O}_{\mathbf{P}^1}(-1)$ and 
$K_Y \sim g^*\mathcal{O}_{\mathbf{P}^1}(-1)$.
If we take $Z = X \times_{\mathbf{P}^1} Y$, then we have $p^*K_X \sim q^*K_Y$.
\end{Rem}

\begin{Que}\label{converse2}
Let $X$ and $Y$ be smooth projective varieties of the same dimension.
Assume that there is a fully faithful functor of triangulated categories
$\Phi: D^b(\text{coh}(X)) \to D^b(\text{coh}(Y))$.
Assume in addition that $\kappa(X,K_X) = \dim X$ or $\kappa(X,-K_X) = \dim X$
Then is it true that $X \le_K Y$, i.e., does there exist a projective variety $Z$ with birational morphisms 
$p: Z \to X$ and $q: Z \to Y$ such that $p^*K_X \le q^*K_Y$?
We note that $S_X \cong \Phi^! \circ S_Y \circ \Phi$ for the right adjoint $\Phi^!$ of $\Phi$.
\end{Que}

%%%%%%%%%%%%%%%%%%%%%
%%%%%%%%%%%%%%%%%%%%%
%%%%%%%%%%%%%%%%%%%%%
%%%%%%%%%%%%%%%%%%%%%%%%%%%%%%%
\section{Factorization of a birational map into elementary ones by MMP}\label{factorizationMMP}

We use the \lq\lq log version'' of the minimal model program, because it is more general and 
useful than the non-log version.
In this case we consider pairs $(X,B)$ consisting of a variety $X$ and an $\mathbf{R}$-divisor $B$,
called a {\em boundary},
with $\mathbf{Q}$-factorial KLT singularities
instead of varieties $X$ with $\mathbf{Q}$-factorial terminal singularities (see definitions below).
If we put $B = 0$, then we are reduced to the non-log case.

\begin{Defn}
Let $X$ be a normal variety and let $B$ be an $\mathbf{R}$-divisor, i.e., a formal linear combination 
$B = \sum_i b_iB_i$ where
the $b_i$ are real numbers and the $B_i$ are mutually distinct prime divisors on $X$.
The pair $(X,B)$ is said to be {\em KLT} (resp. {\em terminal}) if the following condition is satisfied:

\begin{enumerate}

\item The coefficients $b_i$ of $B$ belong to the interval $(0,1)$. 

\item $K_X+B$ is an $\mathbf{R}$-Cartier divisor, a linear combination of Cartier divisors with 
coefficients in $\mathbf{R}$.

\item Let $f: Y \to X$ be a log resolution, a resolution of singularities such that the union of the 
inverse image of the support of $B$ and the exceptional locus 
is a simple normal crossing divisor.
Write $f^*(K_X+B) = K_Y + C$.  
Then the coefficients of $C$ (resp. the coefficients of $C$ except those in the strict transform $f_*^{-1}B$) 
belongs to the interval $(-\infty, 1)$ (resp. $(-\infty, 0)$).
\end{enumerate}

A variety $X$ is said to have terminal singularities if the pair $(X,0)$ is terminal. 
$X$ is {\em $\mathbf{Q}$-factorial} if any prime divisor on $X$ is $\mathbf{Q}$-Cartier, i.e., 
for a prime divisor $D$, there exists a positive integer $m$ such that $mD$ is a Cartier divisor.
\end{Defn}

\begin{Defn}
Let $(X,B)$ and $(Y,C)$ be KLT pairs.
A birational map $f: X \dashrightarrow Y$ is said to be a {\em $K$-equivalence} 
(resp. {\em $K$-inequality}) if 
there is another variety $Z$ with projective birational morphisms
$g: Z \to X$ and $h: Z \to Y$ such that $f = h \circ g^{-1}$ and $g^*(K_X+B) = h^*(K_Y+C)$
(resp. $g^*(K_X+B) \ge h^*(K_Y+C)$), where the canonical divisors $K_X$ and $K_Y$ are defined by using the 
rational differential forms identified each other by $f$.
In this case we write $(X,B) \sim_K (Y,C)$ (resp. $(X,B) \ge_K (Y,C)$).
When $B = C = 0$, we simply write $X \sim_K Y$ or $X \ge_K Y$, etc.
\end{Defn}

We note that the relations $(X,B) \sim_K (Y,C)$ and $(X,B) \ge_K (Y,C)$ are defined 
only if a birational map $f$ is fixed, while $Z$ can be replaced by any other higher model.
We write $(X,B) >_K (Y,C)$ for $(X,B) \ge_K (Y,C)$ and $(X,B) \not\sim_K (Y,C)$.

\vskip 1pc

The elementary birational maps in the MMP are divisorial contractions or flips, 
and they decrease $K_X$ or $K_X+B$.

The factorization by MMP has advantage (monotone) and disadvantage (singularities):

\begin{itemize}
\item The (log) canonical divisor $K_X$ or $K_X+B$ is constantly decreasing.

\item Singularities such as $\mathbf{Q}$-factorial terminal or KLT appear.
\end{itemize}

The following is the definition of an extended family of flips:

\begin{Defn}
A {\em flip} (in a generalized sense) $f: X \dashrightarrow Y$ is a birational map of varieties which satisfies the following conditions:

\begin{enumerate}
\item There is another variety $W$ with projective birational morphisms 
$g: X \to W$ and $h: Y \to W$ such that $f = h^{-1} \circ g$.

\item $g$ and $h$ are small, i.e., isomorphisms in codimension $1$.

\item $g$ and $h$ are elementary in the sense that their relative Picard numbers are $1$; 
$\rho(X/W) = \rho(Y/W) = 1$.
\end{enumerate}
\end{Defn}

We note that, if the morphism $g: X \to W$, called a {\em small contraction}, is given, then 
the other morphism $h: Y \to W$ is uniquely determined if it exists.
This is a consequence of the third condition.
For a given pair $(X,B)$, the choices of contractions are discrete and controlled by 
extremal rays.   

When $X$ is equipped with an $\mathbf{R}$-divisor $B$, we set $C = f_*B$, the strict transform.
We usually assume moreover that $K$ decreases; $(X,B) >_K (Y,C)$.
This case should be called a {\em flip in a strict sense}.
If $K$ stays the same, i.e., $(X,B) \sim_K (Y,C)$, the a flip is called a {\em flop}.
In this case, if we add a small effective $\mathbf{R}$-divisor to $B$, then a flop can be made a flip
in a strict sense.
Moreover a rational map which preserves $K$ and is elementary in certain sense can be called 
a {\em flop} in a generalized sense. 
It often happens that a birational map or even a fiber space can be made to be a flop
by adding a boundary.

The divisorial contractions have similar extended family.

\begin{Defn}
A {\em divisorial contraction} $f: X \to Y$ is a projective birational morphism whose exceptional locus is 
a prime divisor.
\end{Defn}

We set $C = f_*B$, the strict transform. 
We assume usually that $-(K_X+B)$ is $f$-ample, i.e., $(X,B) >_K (Y,C)$.
But such a morphism with $K_X+B$ being $f$-trivial or $f$-ample is also important.
If $K_X+B$ is $f$-trivial, then it can be called a flop in a generalized sense.

The inverse of a divisorial contraction is called a {\em divisorial extraction}.
The direction is reversed from the contraction.

By \cite{BCHM}, there is a {\em $\mathbf{Q}$-factorial terminalization} of a KLT pair:

\begin{Thm}
(1) Let $(X,B)$ be a KLT pair.
Then there exists another KLT pair $(Y,C)$ with a projective birational morphism $f: Y \to X$ which is {\em small},
i.e., an isomorphism in codimension $1$,
such that $Y$ is $\mathbf{Q}$-factorial.
It is automatic that $f^*(K_X+B) = K_Y+C$ in this case.

(2) Let $(X,B)$ be a $\mathbf{Q}$-factorial KLT pair. 
Then there exists a sequence of $\mathbf{Q}$-factorial KLT pairs $\{(X_i,B_i)\}_{0 \le i \le m}$ with 
divisorial extractions $f_i: X_i \to X_{i-1}$ such that $(X,B) = (X_0,B_0)$, $(X_i,B_i) \sim_K (X_{i-1},B_{i-1})$
and $(X_m,B_m)$ is terminal.
\end{Thm}

We state a factorization conjecture of a birational map between $K$-equivalent pairs:

\begin{Conj}\label{factorization of K}
Let $(X,B)$ and $(Y,C)$ be projective $\mathbf{Q}$-factorial terminal pairs.
Assume that $(X,B) \sim_K (Y,C)$ by a birational map $f: X \dashrightarrow Y$.
Then there is a sequence of $\mathbf{Q}$-factorial terminal pairs $\{(X_i,B_i)\}_{0 \le i \le n}$ with $(X_0,B_0) = (X,B)$ 
and $(X_n,B_n) = (Y,C)$ such that 
there are flops $f_i: X_{i-1} \dashrightarrow X_i$ with $(X_{i-1},B_{i-1}) \sim_K (X_i,B_i)$ and $f = f_n \circ \dots \circ f_1$.
\end{Conj}

The following partial answers are known: 

\begin{Thm}
Conjecture~\ref{factorization of K} is true 
in one of the following cases:

\begin{enumerate}

\item $K_X+B$ is nef (\cite{flops-connect}).

\item $f$ is a toric birational map between toric pairs (\cite{toricII}).

\end{enumerate}
\end{Thm}

We can make the conjecture even harder for $K$-inequal pairs:

\begin{Conj}\label{factorization of K2}
Let $(X,B)$ and $(Y,C)$ be projective $\mathbf{Q}$-factorial terminal pairs.
Assume that $(X,B) \ge (Y,C)$ by a birational map $f: X \dashrightarrow Y$.
Then there is a sequence of pairs $\{(X_i,B_i)\}_{0 \le i \le n}$ with $(X_0,B_0) = (X,B)$ and 
$(X_n,B_n) = (Y,C)$ which satisfies the following conditions:

\begin{enumerate}
\item There are birational maps $f_i: X_{i-1} \dashrightarrow X_i$ such that 
$(X_{i-1},B_{i-1}) \ge_K (X_i,B_i)$ and $f = f_n \circ \dots \circ f_1$.  

\item Each $f_i$ is one of the following elementary maps; a divisorial contraction, 
a flip (in a generalized sense) or an isomorphism.

\end{enumerate}
\end{Conj}

\begin{Defn}
Let $D$ be a pseudo-effective $\mathbf{R}$-divisor on a $\mathbf{Q}$-factorial projective variety 
$X$ with an ample divisor $A$.
Let $F = \lim_{\epsilon \to 0} \text{Fix}\Vert D + \epsilon A \Vert$, where 
$\text{Fix}\Vert D + \epsilon A \Vert = \min \{D' \mid D' \equiv D + \epsilon A, D' \ge 0\}$ 
and $\equiv$ denotes the numerical equivalence.
Let $M = D - F$.
We call an expression $D = M + F$ a {\em Zariski decomposition in codimension $1$}.
\end{Defn}

For example, if $D$ is nef, then $F = 0$.

\begin{Lem}\label{preserve}
Let $(X,B)$ and $(Y,C)$ be projective $\mathbf{Q}$-factorial terminal pairs such that $(X,B) \ge_K (Y,C)$.
Let $f: (X,B) \dashrightarrow (X',B')$ be a divisorial contraction or a flip such that 
$B' = f_*B$ and $(X,B) >_K (X',B')$.
Assume that $(Y,C)$ is minimal, i.e., $K_Y+C$ is nef.
Then $(X',B') \ge_K (Y,C)$.
\end{Lem}

\begin{proof}
We take a sufficiently high model $Z$ such that there are projective birational morphisms
$g: Z \to X$, $g': Z \to X'$ and $h: Z \to Y$.
We consider the Zariski decompositions in codimension $1$ on $Z$; 
we have $g^*(K_X+B) = D_1 = M_1+F_1$, $g'{}^*(K_{X'}+B') = D_2 = M_2+F_2$ and $h^*(K_Y+C) = M_3$.
Since $D_1 \ge M_3$, we have $M_1 \ge M_3$.
We denote $D_1 = D_2 + F$ with $F \ge 0$.
Then we have $F_1 \ge F$.  
Hence $D_2 \ge M_3$.
\end{proof}

\begin{Thm}
Under the assumptions of Conjecture~\ref{factorization of K2}, assume in addition the following:

\begin{enumerate}
\item $C = f_*B$, the strict transform.

\item $K_Y+C$ is nef.

\item A minimal model program for the pair $(X,B)$ terminates.
\end{enumerate}
Then the conclusions of \ref{factorization of K2} hold where the $f_i$ are either divisorial contractions or flips.
\end{Thm}

\begin{proof}
We first prove that $f$ is surjective in codimension $1$.
Let $g: Z \to X$ and $h: Z \to Y$ be common resolutions.
Suppose that there is a prime divisor $P$ on $Z$ which is mapped to a prime divisor on $Y$ but not on $X$.
We write $g^*(K_X+B) = K_Z+bP+\dots$ and $h^*(K_Y+C) = K_Z+cP+\dots$.
By $g^*(K_X+B) \ge h^*(K_Y+C)$, we have $b \ge c$.
Since $(X,B)$ is terminal, we have $b < 0$.
But $c \ge 0$, a contradiction.

We run a minimal model program for the pair $(X,B)$.
By Lemma~\ref{preserve}, our conditions are preserved.
Since an MMP terminates, we reach to the situation where $K_X+B$ is nef.
Using the same $Z$ etc. as before, we write $g^*(K_X+B) = h^*(K_Y+C) + E$, where $E$ is effective
and $h_*E = 0$, since $C = f_*B$.
If $E \ne 0$, then there is a curve $l$ on $Z$ such that $h(l)$ is a point and $(E,l) < 0$.
But this contradicts the nefness of $K_X+B$.
Therefore we have $g^*(K_X+B) = h^*(K_Y+C)$.
Then the pairs $(X,B)$ and $(Y,C)$ are connected by flops, flips in a generalized sense, by \cite{flops-connect}.
\end{proof}

%%%%%%%%%%%%%%%%%%%%%
%%%%%%%%%%%%%%%%%%%%%
\section{Varieties of quotient type and associated stacks}\label{quotient}

We explained in the previous section that we should consider not only smooth projective varieties but also
KLT pairs more generally. 
But it is an open question what kind of reasonable \lq\lq derived categories'' should be attached to KLT pairs.
In this section, we will give an answer to this question in the case of pairs of quotient types, a very special case
of KLT pairs. 
It will be justified by Example~\ref{Francia}.

Let $(X,B)$ be a pair of a normal variety and a $\mathbf{Q}$-divisor.
It is said to be of {\em quotient type} if there exists a smooth scheme $U$, which may be reducible, 
with a quasi-finite and surjective morphism $\pi: U \to X$ such that  
$\pi^*(K_X+B)=K_U$.
The ramification of $\pi$ in codimension $1$ is given by $B$.
In particular the coefficients of $B$ are standard, i.e., belong to a set $\{1 - 1/n \mid n \in \mathbf{Z}_{>0}\}$.

Let $R = (U \times_X U)^{\nu}$ be the normalization of the fiber product.
Then $R$ is smooth and there is a natural finite morphism $R \to U \times U$, defining a 
Deligne-Mumford stack $\tilde X$.
There is a natural birational and bijective morphism $\pi_X: \tilde X \to X$.

Since $U$ is smooth, $\tilde X$ is smooth, so that the derived category $D^b(\text{coh}(\tilde X))$ has
finite projective dimension.
The following is our {\em $KD$-hypothesis} in this case:

\begin{Conj} 
Let $(X,B)$ and $(Y,C)$ be pairs of quotient type such that $X$ and $Y$ are projective, and let 
$\tilde X$ and $\tilde Y$ be the associated smooth Deligne-Mumford stacks.

\begin{enumerate}
\item If $(X,B) \sim_K (Y,C)$, i.e., there are projective birational morphisms
$f: Z \to X$ and $g: Z \to Y$ such that $f^*(K_X+B) = g^*(K_Y+C)$, 
then there is an equivalence of triangulated categories
$\Phi: D^b(\text{coh}(\tilde X)) \cong D^b(\text{coh}(\tilde Y))$.

\item If $(X,B) \le_K (Y,C)$, i.e., there are projective birational morphisms
$f: Z \to X$ and $g: Z \to Y$ such that $f^*(K_X+B) \le g^*(K_Y+C)$, 
then there is a fully faithful functor of triangulated categories
$\Phi: D^b(\text{coh}(\tilde X)) \to D^b(\text{coh}(\tilde Y))$.
\end{enumerate}
\end{Conj}

The conjecture above is generalized so that the DK hypothesis for varieties as before and the derived 
McKay correspondence for quotient singularities are both included (see \S \ref{GL3}).

This definition of the derived category of $(X,B)$ is justified by the following example
(this is a special case of toric flops explained in \S 8):

\begin{Expl}[Francia flop \cite{Francia}]\label{Francia}
Let $X$ be a smooth $4$-fold which is the total space of a vector bundle
$\mathcal{O}_{\mathbf{P}^2}(-1) \oplus \mathcal{O}_{\mathbf{P}^2}(-2)$ over $\mathbf{P}^2$.
The zero section $E \cong \mathbf{P}^2 \subset X$ can be contracted 
by a projective birational morphism $f: X \to Y$ to an isolated singular point $Q \in Y$ 
of a normal Gorenstein affine variety.
$f$ is a small contraction such that $K_X = f^*K_Y$.
There is a flop $f^+: X^+ \to Y$ of $f$, whose the exceptional locus $E^+ = f^+{}^{-1}(Q)$ is 
isomorphic to $\mathbf{P}^1$.

$X^+$ has one isolated singularity $P \in E^+$, which is a quotient singularity of type 
$\frac 12(1,1,1,1)$, i.e., the singularity of a quotient $\mathbf{C}^4/\mathbf{Z}_2$ by the involution $x \mapsto -x$.
Let $\tilde X^+$ be the smooth Deligne-Mumford stack associated to $(X^+, 0)$.
Then there is an equivalence $D^b(\text{coh}(X)) \cong D^b(\text{coh}(\tilde X^+))$.
But the derived categories $D^b(\text{coh}(X))$ and $D^b(\text{coh}(X^+))$ are very different.
For example, for the point object $\mathcal{O}_P \in D^b(\text{coh}(X^+))$, 
we have $\text{Hom}(\mathcal{O}_P,\mathcal{O}_P[k]) \ne 0$ for all $k \ge 0$, but 
there is no such object in $D^b(\text{coh}(X))$.

This example can be regarded as a mixture of a smooth flop with a McKay correspondence 
(without a crepant resolution).
\end{Expl}

Orlov's representability theorem is generalized to the case of varieties of quotient types in \cite{equi}.
Hence we have the following consequence similarly to Theorem~\ref{converse1}:

\begin{Thm}
Let $(X,B)$ and $(Y,B)$ be pairs of quotient types such that $X$ and $Y$ are projective varieties.
Let $\tilde X$ and $\tilde Y$ be the associated smooth Deligne-Mumford stacks.
Assume that there is an equivalence of triangulated categories 
$\Phi: D^b(\text{coh}(\tilde X)) \cong D^b(\text{coh}(\tilde Y))$.
Assume in addition that $\kappa(X,K_X+B) = \dim X$ or $\kappa(X,-K_X-B) = \dim X$.
Then there exists a projective scheme $Z$ and birational morphisms
$f: Z \to X$ and $g: Z \to Y$ such that $f^*(K_X+B) = g^*(K_Y+C)$.
\end{Thm}

%%%%%%%%%%%%%%%%%%%%%
%%%%%%%%%%%%%%%%%%%%%
%%%%%%%%%%%%%%%%%%%%%
%%%%%%%%%%%%%%%%%%%%%%%%
\section{Major approaches}\label{approach}

We review major approaches found so far toward DK hypothesis in the case of flops
between smooth varieties.
The mathematics in each of them is quite interesting in its own right.

Let $X$ and $Y$ be smooth varieties with projective birational morphisms $f: X \to Z$ and $g: Y \to Z$
which are small and such that the relative Picard numbers $\rho(X/Z) = \rho(Y/Z) = 1$ and that $K_X \sim_K K_Y$. 
We would like to prove that there is an equivalence $\Phi: D^b(\text{coh}(X)) \to D^b(\text{coh}(Y))$.
There are several general approaches to the conjecture:

\begin{enumerate}

\item (moduli) We represent $X$ as a certain moduli space of objects $M_x \in D^b(\text{coh}(Y))$
for $x \in X$, and use this structure to prove the conjecture.

\item (VGIT) We represent $X$ and $Y$ as different geometric quotients $V//_-G$ and $V//_+G$ 
of the same geometric invariant theory (GIT) problem,  
find a full subcategory $\mathcal{W}$ of the big quotient $D^b(\text{coh}([V/G]))$, called a {\em window}, 
and prove that the natural projections from $\mathcal{W}$ 
to $D^b(\text{coh}(X))$ and $D^b(\text{coh}(Y))$ are equivalences.

\item (tilting) We construct {\em tilting generators} $M \in D^b(\text{coh}(X))$ and $N \in D^b(\text{coh}(Y))$
whose endomorphism rings are isomorphic $\text{End}(M) \cong \text{End}(N) \cong R$.
Then 
\[
D^b(\text{coh}(X)) \cong D^b(\text{mod-}R) \cong D^b(\text{coh}(Y)).
\]

\end{enumerate}

(1) {\em moduli method}. 
This approach is based on the Orlov representability theorem (Theorem~\ref{Orlov}) stating that 
$\Phi(\bullet) \cong p_{2*}(p_1^*(\bullet) \otimes^L e)$.
$X$ is regarded as a moduli space of objects 
$\Phi(\mathcal{O}_x) = M_x = p_{2*}(\mathcal{O}_{x \times Y} \otimes^L e) \in D^b(\text{coh}(Y))$, 
with the kernel $e$ being the universal family. 
In this way, $X$ becomes naturally a stack.

This method gives the most accurate description of the change of categories,
but it seems to work only in the case where we know the situation very well,
e.g., in dimension 2 or 3.
Since the Fourier-Mukai kernel $e$ is given as a universal object, 
the equivalence is naturally global.

The method is discovered in \cite{Bridgeland2} and \cite{BKR}, and extended in \cite{Chen}.
In \cite{Bridgeland2}, we define an abelian category of {\em perverse coherent sheaves} 
by gluing the abelian category of coherent sheaves on the contracted space $Z = f(X)$
with the heart of the category of acyclic objects
$\mathcal{C} = \{c \in D^b(\text{coh}(X)) \mid f_*c = 0\}$ with shift by $-1$.
A {\em perverse point sheaf} is defined to be a quotient of $\mathcal{O}_X$ in this abelian category 
which is numerically equivalent to 
the ordinary point sheaf $\mathcal{O}_x$.
Then $Y$ is the moduli space of such perverse point sheaves.

Similarly in \cite{BKR}, a crepant resolution of a quotient singularity $X = \mathbf{C}^3/G$ by a finite subgroup 
$G \subset SL(3,\mathbf{C})$
is constructed as a moduli space $Y$, called a {\em $G$-Hilbert scheme}, which classifies {\em $G$-clusters}, 
those subschemes of $\mathbf{C}^3$ which are finite dimensional as $G$-modules and 
numerically equivalent to the structure sheaves of free $G$-orbits.
Then $D^b(\text{coh}(Y)) \cong D^b(\text{coh}(\tilde X))$, where 
$\tilde X = [\mathbf{C}^3/G]$ is the Deligne-Mumford stack associated to the quotient space $X$ 
(see \S \ref{quotient}).

\cite{Toda-MMP2} and \cite{Toda-MMP3}
exhibited the MMP for surfaces and the first step for 3-folds as moduli spaces of stable objects
with respect to some stability conditions.
The wall crossings of stability conditions in this case correspond to the change of birational models
with decreasing canonical divisors

\vskip 1pc

The following definition is similar to that of the tilting generator below:

\begin{Defn}
A class of objects $\Omega \subset D^b(\text{coh}(X))$ is said to be a {\em spanning class}
if the following conditions are satisfied:

\begin{enumerate}
\item For an object $a \in D^b(\text{coh}(X))$, $\text{Hom}(\omega,a[p]) = 0$ 
for all $\omega \in \Omega$ and for all $p \in \mathbf{Z}$
implies that $a \cong 0$.

\item For an object $a \in D^b(\text{coh}(X))$, $\text{Hom}(a,\omega[p]) = 0$ 
for all $\omega \in \Omega$ and for all $p \in \mathbf{Z}$
implies that $a \cong 0$.
\end{enumerate}
\end{Defn}

\begin{Thm}[\cite{Bridgeland1}]
Let $F: \mathcal{A} \to \mathcal{B}$ be an exact functor of triangulated categories which have Serre functors.

(1) Assume that $F: \text{Hom}(\omega, \omega'[p]) \to \text{Hom}(F(\omega), F(\omega')[p])$ are 
bijective for all $\omega,\omega' \in \Omega$ and for all $p \in \mathbf{Z}$.
Then $F$ is fully faithful.

(2) Assume that $F$ is fully faithful, and that $S_{\mathcal{B}} \circ F \cong F \circ S_{\mathcal{A}}$
for Serre functors.
Then $F$ is an equivalence.
\end{Thm}

For example, the set of all point objects $\Omega = \{\mathcal{O}_x \mid x \in X\}$ for 
a smooth projective variety $X$ is a spanning class of the bounded derived category 
$D^b(\text{coh}(X))$.
As a consequence, we can check the fully-faithfulness of an exact functor \lq\lq pointwise'', 
hence analytic locally. 

Therefore, when $X$ is the moduli space of objects $M_x \in D^b(\text{coh}(Y))$, the key point is to 
prove that the natural homomorphisms
\[
\text{Hom}(\mathcal{O}_x,\mathcal{O}_{x'}[p]) \cong \text{Hom}(M_x,M_{x'}[p])
\]
are bijective for all $x,x'$ and $p$, in order to prove the derived equivalence
$D^b(\text{coh}(X)) \cong D^b(\text{coh}(Y))$.

Since the Serre functor is intrinsically determined 
for a triangulated category, the equivalence of derived categories 
implies the compatibility of the Serre functors.
The second assertion of the above theorem is a converse.

\vskip 1pc 

(2) {\em VGIT method}.
This approach of using the variation of geometric invariant theory (VGIT) was discovered in a physics paper \cite{Hori}.
\cite{DS} treated {\em stratified Mukai flops} of type $A$, flops between different Springer resolutions
yielding the cotangent bundles of dual Grassmannian varieties.
\cite{BFK} and \cite{HL} developed more general theory.
\cite{BFK} treats also the {\em Landau-Ginzburg models} which have additional structure to the varieties.

If there is an algebraic group $G$ acting on the whole situation $V$, then there is a possibility of using this method.
We represent $X$ and $Y$ as different geometric quotients $V//_-G$ and $V//_+G$ 
of the same GIT problem.
We find a full subcategory $\mathcal{W}$ of the derived category of the big quotient stack $D^b(\text{coh}([V/G]))$, 
called a {\em window}.
This is determined by restricting the weights along a Kirwan-Ness stratum of the unstable locus.
Then we prove that the natural functor from $D^b(\text{coh}([V/G]))$ 
to $D^b(\text{coh}(X))$ and $D^b(\text{coh}(Y))$ induce equivalences from $\mathcal{W}$.

The method works well along with the representation theory of algebraic groups.
But even in the case of stratified Mukai flops of type $B$, singularities appear in this construction, 
and similar arguments do not work,  
because the general theory requires smoothness of the whole situation. 

Any change of birational models can be regarded as a VGIT (\cite{Thaddeus}, \cite{Dolgachev-Hu}),
e.g., weak factorization theorem in \S \ref{weak fact}.
We have smoothness in this case, but there are too many Kirwan-Ness strata corresponding to the up-and-down
of the canonical divisor. 

\vskip 1pc

(3) {\em tilting method}.
This method was developed by
Van den Bergh \cite{VdBergh}, where 
an alternative proof of \cite{Bridgeland2} on 3-fold flops is given
in more general situation.
The equivalence of the derived categories of two varieties are proved 
through the third derived category of a non-commutative associative algebra.
Namely, we find tilting generators $M$ and $N$ (defined below) on $X$ and $Y$, respectively, 
whose endomorphism rings $R$ are isomorphic: $\text{End}(M) \cong \text{End}(N) \cong R$.
Then we have derived equivalences 
\[
D^b(\text{coh}(X)) \cong D^b(\text{mod-}R) \cong D^b(\text{coh}(Y)).
\]
In this sense we leave the world of commutative algebras and schemes.
This method looks working most generally, and provides the easiest way 
to prove the equivalence of triangulated categories.
\cite{Kaledin} and \cite{Segal} use this method. 

The kernel of the equivalence is not clearly given by this method.
Therefore the globalization, i.e., gluing the derived equivalences given locally over $Z$ to the whole situation, 
is not immediate.

The non-commutative algebra obtained in the proof is also related to the geometric structure of the flop
which was not captured by traditional method, 
e.g., non-commutative deformations of the exceptional curve (\cite{DW}).

\begin{Defn}
An object $M \in D^b(\text{coh}(X))$ is said to be a {\em tilting generator}
if the following conditions are satisfied:

\begin{enumerate}
\item For an object $a \in D^b(\text{coh}(X))$, $\text{Hom}(M,a[p]) = 0$ for all $p \in \mathbf{Z}$
implies that $a \cong 0$.

\item $\text{Hom}(M,M[p]) \cong 0$ for all $p \ne 0$. 
\end{enumerate}
\end{Defn}

By using the tilting generator, we can connect different worlds of varieties and non-commutative rings.

\begin{Thm}[\cite{Bondal}, \cite{Rickard}]\label{Rickard}
Let $R = \text{End}(M)$.
Then $D^b(\text{coh}(X)) \cong D^b(\text{mod-}R)$.
\end{Thm}

In the case of a $3$-fold flopping contraction $f: X \to Z$, a tilting generator $M \in D^b(\text{coh}(X))$ 
is constructed as a universal extension of lines bundles which generate the relative Picard group.
Then $R = \text{End}(M)$ is a non-commutative algebra over $Z$.
The point of the proof is to show that another algebra over $Z$ constructed from the other side $g: Y \to Z$ is
isomorphic to $R$. 

\vskip 1pc

There are other approaches in special cases.
In particular, the approach by Cautis \cite{Cautis} is interesting, 
which gives a proof of equivalence in the case of stratified Mukai flops.
This approach is related to the VGIT approach, but gives more precise information on the kernels.
The proof gives the kernel as a convolution of a complex of objects.
The proof looks like a kind of derived version of projective geometry. 
There is also \cite{Kuznetsov-IMOU}.

We shall explain the case of toroidal birational morphisms in \S \ref{toric},
where we shall even prove the full $K$-inequality conjecture by explicit combinatorial method.
In the toric case, the proof is given by tilting generators.
But the kernel is given by the fiber product. 
Therefore the functors glue globally, so that the statement is 
generalized to the toroidal case.
In order to apply the results to the McKay correspondence for $GL(3)$, we need to use the toroidal version.

%%%%%%%%%%%%%%%%%%%%%
%%%%%%%%%%%%%%%%%%%%%
%%%%%%%%%%%%%%%%%%%%%
\section{Weak factorization of a birational map}\label{weak fact}

We consider a factorization of a birational map between smooth projective varieties into blow-ups and blow-downs
with smooth centers in this section.
The advantage of this weak factorization is that all varieties appearing in the intermediate steps 
are smooth.
On the other hand, the canonical divisor changes up and down without fixed direction.

Since any birational map $f: X \dashrightarrow Y$ is factorized 
into two birational morphisms, i.e.. there is a third smooth projective variety $Z$ with birational morphisms
$g: Z \to X$ and $h: Z \to Y$ such that $f = h \circ g^{-1}$, we only need to factorize birational morphisms.
In dimension $2$, any birational morphism is factorized into a sequence of blow-downs, i.e., in one direction.
But in higher dimensions, we need up and down.

We shall need the logarithmic version of the factorization theorem.
We consider pairs $(X,B)$ instead of varieties $X$.

Let $X$ be a smooth variety and let $\sum B_i$ be a simple normal crossing divisor, i.e., irreducible
components $B_i$ are smooth and cross normally.
A blowing up $f: Y \to X$ along a smooth center $C \subset X$ is said to be {\em permissible} 
for the pair $(X, \sum B_i)$ if $C$ is normal
crossing with $\sum B_i$, i.e. there exist local coordinates $(x_1,\dots,x_n)$ at each point $x \in X$ 
such that $\sum B_i$ and $C$ are expressed as 
$x_1 \dots x_r = 0$ and $x_1 = \dots = x_s = x_{r+1} = \dots = x_{r+t} = 0$ respectively 
for some $r,s,t$ such that $0 \le s \le r$ and $0 \le t$.
We have $\text{codim}(C) = s+t$.
In this case, the union of the set theoretic inverse image 
$f^{-1}(\sum B_i)$ and the exceptional divisor is again a simple normal crossing divisor.

The following theorem was proved using the theory of algebraic Morse theory 
developed in \cite{W-Morse}:

\begin{Thm}[weak factorization theorem \cite{W}, \cite{AKMW}]\label{Morse}
Let $f: Y \to X$ be a birational morphism between smooth projective varieties.
Assume that there are simple normal crossing divisors $B$ and $C$ on $X$ and $Y$, respectively, 
such that $Y \setminus C \cong X \setminus B$ by $f$.
Then there exists a sequence of birational maps $\phi_i: Y_{i-1} \dashrightarrow Y_i$ 
for $i = 1,\dots, n$ with $Y = Y_0$ and $X = Y_n$ such that the following hold:

\begin{enumerate}

\item $f_i:= \phi_n \circ \dots \phi_{i+2} \circ \phi_{i+1}: Y_i \to X$ are morphisms for $0 \le i \le n$, and 
$f = f_0$.

\item The $Y_i$ are smooth projective varieties for all $i$,
$C_i = f_i^{-1}(B)$ is a simple normal crossing divisor on $Y_i$ set theoretically, and 
$Y_i \setminus C_i \cong X \setminus B$ by $f_i$.

\item $\phi_i$ or $\phi_i^{-1}$ is a permissible blowing up along a smooth center contained in $C_i$ or 
$C_{i-1}$, respectively, for each $i$.
\end{enumerate}
\end{Thm}

The idea of the proof is as follows.
We express $f$ as a blowing up of an ideal sheaf $I \subset \mathcal{O}_X$ such that 
$\text{codim}(\text{Supp}(\mathcal{O}_X/I)) \ge 2$.
Let $F: W \to V = X \times \mathbf{P}^1$ be a blowing up of an ideal sheaf 
$\tilde I = (p_1^*I, p_2^*I_0) \subset \mathcal{O}_V$, 
where $I_0 \subset \mathcal{O}_{\mathbf{P}^1}$ is the ideal sheaf of $0 \in \mathbf{P}^1$.
An algebraic group $G = \mathbf{G}_m$ acts naturally on $\mathbf{P}^1$, hence on $V$.
$G$ acts on $W$ such that $F$ is $G$-equivariant.
The strict transforms of $X \times \{0\}$ and $X \times \{\infty\}$ by $F$ are isomorphic to $Y$ and $X$, 
respectively, and the variation of GIT quotients determines $f$ (\cite{Thaddeus}, \cite{Dolgachev-Hu}).
There is a suitable $G$-equivariant resolution of singularities $\mu: \tilde W \to W$.
The algebraic Morse theory is the variation of GIT quotients along the flow of the $G$-action on $\tilde W$,
and yields the factorization of $f$ into elementary birational maps $\phi_i$.
Since there is a natural morphism $\tilde W \to X$, 
all intermediate varieties $Y_i$ have morphisms to $X$.

%%%%%%%%%%%%%%%%%%%%%
%%%%%%%%%%%%%%%%%%%%%
%%%%%%%%%%%%%%%%%%%%%
\section{Application to the Grothendieck rings of varieties and categories}\label{K0}

We shall prove that, for $K$-equivalent smooth projective varieties, the classes of their derived categories 
in the Grothendieck ring of categories with $\mathbf{Q}$-coefficients are equal.
This is a weak evidence of the DK hypothesis.

Let $K_0(\text{Var}_k)$ be the Grothendieck ring of varieties over $k$.
It is a commutative ring generated by all isomorphism classes of varieties $X$ defined over $k$
with relations generated by $[X] = [X \setminus Z] + [Z]$ for subvarieties $Z \subset X$, where 
$[X]$ denotes the class of $X$ in $K_0(\text{Var}_k)$.
The product is defined by the direct product $[X] [Y] = [X \times Y]$.
Let $L$ be the class of an affine line, $L = [\mathbf{A}^1]$. 
Then the class of a projective space satisfies an equality 
\[
[\mathbf{P}^b] = \sum_{i = 0}^b L^i = (L^{b+1}-1)/(L-1)
\]
in $K_0(\text{Var}_k)$.

\cite{LaLu} proved that the quotient $K_0(\text{Var})/(L) \cong \mathbf{Z}[SB]$, the ring generated by the 
stable birational classes of varieties. 
\cite{Kuznetsov-Shinder} and \cite{IMOU} independently conjectured that  
$D$-equivalence implies $L$-equivalence: 
if $D^b(\text{coh}(X)) \cong D^b(\text{coh}(Y))$, then $[X] = [Y] \in K_0(\text{Var})[L^{-1}]$.
We note that the method of motivic integration yields an equality in the completion:
$[X] = [Y] \in \lim_{m \to \infty} K_0(\text{Var}/k)[L^{-1}]/\{z \mid \dim z < -m\}$.

Let $(X,B)$ be a pair consisting of a smooth variety $X$ and a divisor with real coefficients 
$B = \sum_{i \in I} b_i B_i$
such that its support $\sum B_i$ is a simple normal crossing divisor and such that $b_i < 1$ for all $i$.
We define a {\em stringy invariant} of the pair
\[
S(X,B) = \sum_{J \subset I} [B_J^o] \prod_{j \in J} (L-1)/(L^{1-b_j}-1)
\]
where $J$ runs all subsets of the set of indexes $I = \{i\}$, and we denote
$B_J = \bigcap_{j \in J} B_j$ and $B_J^o = B_J \setminus \bigcup_{J \subsetneq J'} B_{J'}$ 
following Batyrev \cite{Batyrev}.
We understand that $B_{\emptyset} = X$.
$S(X,B)$ is an element in a ring which is obtained from $K_0(\text{Var}_k)$ 
by attaching some functions of $L$. 
If all $b_i$ are integers, then it is an element in a localized ring $K_0(\text{Var}_k)([\mathbf{P}^n]^{-1} \mid n > 0)$.

\begin{Lem}\label{permissible}
Let $(X,B)$ be a pair consisting of a smooth variety $X$ and a divisor with real coefficients 
$B = \sum_{i \in I} b_i B_i$
such that its support $\sum B_i$ is a simple normal crossing divisor and such that $b_i < 1$ for all $i$.
Let $f: Y \to X$ be a permissible blowing up of the pair $(X,\sum B_i)$ along a smooth center $C$.
Let $B_Y = f^*B - (c - 1)E$ for $c = \text{codim}(C)$ and the exceptional divisor $E$, 
so that $f^*(K_X + B) = K_Y + B_Y$.
Then $S(X,B) = S(Y,B_Y)$.
\end{Lem}

For example, if $B = 0$ and $B_Y = (1-c)E$, then
$S(X,B) = [X]$ and $S(Y,B_Y) = [X \setminus C] + [E]/[\mathbf{P}^{c-1}] = [X]$, 
because $E$ is a $\mathbf{P}^{c-1}$-bundle over $C$.

\begin{proof}
We may assume that $C$ is contained in $B_i$ for only $i = 1,\dots,s$ and $c = s + t$ for some $s, t \ge 0$.
Then $B_Y = \sum b_iB'_i + (\sum_{i=1}^s b_i - c + 1)B'_0$, where 
$B'_i = f_*^{-1}B_i$ is the strict transform of $B_i$ and $B'_0 = E$. 
Since $c \ge s$, the coefficients of $E$ in $B_Y$ is less than $1$.
We set $B' = \sum_{i \in I'} B'_i$ for $I' = I \cup \{0\}$.

Let $B_J^o$ be any stratum of $X$.
If $B_J^o \cap C = \emptyset$, then $B_J^o \cong (B'_J)^o$.
Then the contributions of $B_J^o$ to $S(X,B)$ and $(B'_J)^o$ to $S(Y,B_Y)$ are equal.

We consider the case $B_J^o \cap C \ne \emptyset$ in the following.
Then it follows that $\{1,\dots,s\} \subset J$.
We set $J = \{1,\dots,s\} \coprod J'$.
If $t \ne 0$, then the strata of $Y$ above $B_J^o$ are $(B'_J)^o$ and the $(B'_{K'})^o$ for
$K' = K \coprod J' \coprod \{0\}$ where $K$ runs all subsets $K \subset \{1,\dots,s\}$.
If $t = 0$, then the same holds except that $K \ne \{1,\dots,s\}$.
We have $[(B'_J)^o] = [B_J^o] - [B_J^o \cap C]$, and 
\[
[(B'_{K'})^o] = \begin{cases}
[B_J^o \cap C] [\mathbf{A}^1 \setminus \{0\}]^{s-1-\# K}[\mathbf{A}^t] \,\, \text{ if } \# K < s \\
[B_J^o \cap C] [\mathbf{P}^{t-1}] \,\, \text{ if } \# K = s. 
\end{cases}
\]
Then the corresponding contribution of the $(B'_{K'})^o$ to $S(Y,B_Y)$ is:
\[
\begin{split}
[B_J^o &\cap C] \prod_{j \in J'} (L-1)/(L^{1-b_j}-1) \times (L-1)/(L^{c-\sum_{i=1}^s b_i} -1) \\
\times (&\sum_{0 \le k < s} \sum_{1 \le i_1 < \dots i_k \le s}  (L-1)^{s-1-k}L^t \prod_{j=1}^k (L-1)/(L^{1-b_{i_j}}-1) \\
&+ (L^t - 1)/(L-1) \prod_{i = 1}^s (L-1)/(L^{1-b_i}-1)) 
\end{split}
\]
where we should ignore the term of the third line if $t = 0$.
The sum of the terms in the second and third lines is equal to: 
\[
\begin{split}
&(L-1)^{s-1} (\sum_{0 \le k < s} \sum_{1 \le i_1 < \dots i_k \le s}  L^t \prod_{j=1}^k 1/(L^{1-b_{i_j}}-1) 
+ (L^t - 1) \prod_{i = 1}^s 1/(L^{1-b_i}-1)) \\
&=(L-1)^{s-1} (L^t (\prod_{i=1}^s (1 + 1/(L^{1-b_i}-1)) - \prod_{i=1}^s 1/(L^{1-b_i}-1))
+ (L^t - 1) \prod_{i = 1}^s 1/(L^{1-b_i}-1)) \\
&=(L-1)^{s-1}(L^{t+s-\sum_{i=1}^s b_i} \prod_{i=1}^s 1/(L^{1-b_i}-1) - \prod_{i = 1}^s 1/(L^{1-b_i}-1)).
\end{split}
\]
Therefore the total contribution of the strata above $B_J^o$ to $S(Y,B_Y)$ is:
\[
\begin{split}
&([B_J^o] - [B_J^o \cap C] + [B_J^o \cap C]) \prod_{j \in J'} (L-1)/(L^{1-b_j}-1) \prod_{i=1}^s (L-1)/(L^{1-b_i}-1) \\
&= [B_J^o] \prod_{j \in J} (L-1)/(L^{1-b_j}-1)
\end{split}
\]
and we conclude the proof.
\end{proof}

\begin{Thm}\label{S(X,B)}
Let $(X,B)$ and $(Y,C)$ be pairs consisting of smooth projective varieties and $\mathbf{R}$-divisors
whose supports are simple normal crossing divisors and whose coefficients are less than $1$.
Assume that there are birational morphisms $f: Z \to X$ and $g: Z \to Y$ from another smooth projective
variety such that $f^*(K_X+B) = g^*(K_Y+C)$.
Then $S(X,B) = S(Y,C)$.
\end{Thm}

\begin{proof}
We denote $\bar B = \text{Supp}(B)$ and $\bar C = \text{Supp}(C)$.
We take open dense subsets $U \subset X$ and $V \subset Y$ such that 
$U \cap \bar B = \emptyset = V \cap \bar C$, and 
$U \cong f^{-1}(U) = g^{-1}(V) \cong V$ by $f$ and $g$.
By Hironaka's resolution theorem, there are sequences of birational morphisms
$f_i: X_i \to X_{i-1}$ for $i = 1,\dots,l$ and 
$g_j: Y_j \to Y_{j-1}$ for $j=1,\dots,m$ such that the following conditions are satisfied:

\begin{enumerate}
\item $X_0 = X$ and $Y_0 = Y$.

\item There are simple normal crossing divisors $\bar B^i$ on $X_i$ (resp. $\bar C^j$ on $Y_j$)
such that $\bar B^0 = \bar B$ (resp. $\bar C^0 = \bar C$),  
and such that $\bar B^i$ (resp. $\bar C^j$) is the union of the strict transform $f_{i*}^{-1}\bar B^{i-1}$
(resp. $g_{j*}^{-1}\bar C^{j-1}$) and the exceptional divisor of $f_i$ (resp. $g_j$).

\item $f_i$ (resp. $g_j$) is a blowing up along a smooth center which is permissible for $\bar B^{i-1}$
(resp. $C^{j-1}$).

\item $\bar B^l = X_l \setminus (f_1 \circ \dots \circ f_l)^{-1}(U)$ and 
$\bar C^m = Y_m \setminus (g_1 \circ \dots \circ g_m)^{-1}(V)$.
\end{enumerate}

We note that, if there are codimension $1$ irreducible components 
of the complements $X \setminus \bar B_l$ and $Y \setminus \bar B_m$, 
then some of the $f_i$ and $g_j$ are identities 
which are blowing up along smooth divisors.   

Let $X' = X_l$ and $Y' = Y_m$.
We define $\mathbf{R}$-divisors $B'$ and $C'$ on $X'$ and $Y'$, respectively, 
such that $(f_1 \circ \dots \circ f_l)^*(K_X+B) = K_{X'} + B'$
and $(g_1 \circ \dots \circ g_m)^*(K_Y+C) = K_{Y'}+C'$.
The supports of $B'$ and $C'$ are contained in $\bar B^l$ and $\bar C^m$, respectively, and
their coefficients are less than $1$.

There are birational morphisms $f': Z' \to X'$ and $g': Z' \to Y'$ from another smooth projective
variety such that $(f')^*(K_{X'}+B') = (g')^*(K_{Y'}+C')= K_{Z'} + D'$ for an $\mathbf{R}$-divisor $D'$ 
on $Z'$ whose support is a normal crossing divisor and whose coefficients are less than $1$.
Then we have
\[
S(X,B) = S(X',B') = S(Z',D') = S(Y',C') = S(Y,C)
\]
where the first and the last equalities are consequences of Lemma~\ref{permissible}, and 
the second and the third of Lemma~\ref{permissible} and Theorem~\ref{Morse}.
Thus we proved the theorem.
\end{proof}

By using the motivic integration, we obtain a similar result but in a bigger ring, the completion of 
the localization $K_0(\text{Var}_k)[L^{-1}]$ with respect to the degree (\cite{Batyrev}).

\begin{Cor}
Let $X$ and $Y$ be smooth projective varieties over a field $k$ of characteristic $0$.
Assume that $X$ and $Y$ are $K$-equivalent.
Then $[X] = [Y]$ in $K_0(\text{Var}_k)([\mathbf{P}^n]^{-1} \mid n > 0)$.
\end{Cor}

Bondal-Larsen-Lunts defined a Grothendieck ring $K_0(\text{Cat})$ for derived categories 
\cite{Bondal-Larsen-Lunts}.
They considered pretriangutated DG categories, i.e., enhanced triangulated categories, as generators and 
semi-orthogonal decompositions of triangulated categories as relations: 
if $\mathcal{A} = \langle \mathcal{B}, \mathcal{C} \rangle$, then 
$[\mathcal{A}] = [\mathcal{B}] + [\mathcal{C}]$ in $K_0(\text{Cat})$.

For a smooth projective variety $X$, we can consider its class $[D^b(\text{coh}(X))]$ in $K_0(\text{Cat})$.
$[D^b(\text{coh}(\text{Spec }k))]$ is the unit of the ring $K_0(\text{Cat})$.
By \cite{Beilinson}, we have $[D^b(\text{coh}(\mathbf{P}^n))] = n+1$.
Therefore we have the following corollary:

\begin{Cor}
Let $X$ and $Y$ be smooth projective varieties over a field $k$ of characteristic $0$.
Assume that $X$ and $Y$ are $K$-equivalent.
Then $[D^b(\text{coh}(X))] = [D^b(\text{coh}(Y))] \in K_0(\text{Cat}) \otimes \mathbf{Q}$.
\end{Cor}

We remark that, if $X$ is a Calabi-Yau manifold, then $D^b(\text{coh}(X))$ has no semi-orthogonal decomposition.
Thus we have a weak supporting fact for a conjecture that, if $X$ and $Y$ are birationally equivalent 
Calabi-Yau manifolds, then $D^b(\text{coh}(X))$ and $D^b(\text{coh}(Y))$ are equivalent. 

%%%%%%%%%%%%%%%%%%%%%
%%%%%%%%%%%%%%%%%%%%%
\section{Toric and toroidal cases}\label{toric}

We review results on DK hypothesis for toric and toroidal varieties proved in 
\cite{log crep}, \cite{toric}, \cite{toricII} and \cite{toricIII}.
We proved the conjecture for toric birtaional maps of elementary type, i.e, 
divisorial contractions and flips, not only for $K$-equivalences but also for $K$-inequalities.
We constructed kernels of fully faithful functors as the structure sheaves of the fiber products.

Toroidal varieties are those which are etale locally isomorphic to toric varieties.
The results on toric morphisms are extended to toroidal ones.
Indeed, since the kernels are constructed globally by using the fiber product, 
the fully faithfulness of functors can be checked locally, i.e.,
the assertions are glued globally.
We need the toroidal version of the theorem in 
the application to the McKay correspondence for finite subgroups of $GL(3,\mathbf{C})$ treated in \S \ref{GL3}.

We consider $\mathbf{Q}$-factorial toric pairs with standard coefficients in this section:

\begin{Lem}
Let $X$ be a normal toric variety and let $B$ be a toric $\mathbf{Q}$-divisor.
The pair $(X,B)$ is of quotient type if and only if $X$ is $\mathbf{Q}$-factorial and $B$ has 
standard coefficients, i.e., they belong to the set $\{1-1/n \mid n \in \mathbf{Z}_{>0}\}$.
\end{Lem}

A quotient singularity by a finite abelian group is toric. 
A toric variety is $\mathbf{Q}$-factorial if and only if it has only abelian quotient singularities. 
$\mathbf{Q}$-factorial toric pair is always KLT.

\begin{Defn}
A sequence of objects $(e_1,\dots,e_n)$ of a triangulated category $\mathcal{A}$ is said to be an
{\em exceptional collection} if the following condition is satisfied:
\[
\text{Hom}(e_i,e_j[p]) = \begin{cases} k \,\, &\text{ if } i = j, p = 0 \\
0 \,\, &\text{ if } i = j, p \ne 0 \\
0 \,\, &\text{ if } i > j, \forall p.
\end{cases}
\]
It is called {\em strong} if $\text{Hom}(e_i,e_j[p]) = 0$ for all $i,j$ and $p \ne 0$.
It is said to be {\em full}, or {\em generate} $\mathcal{A}$ if, for any $a \in \mathcal{A}$, 
$\text{Hom}(e_i,a[p]) = 0$ for all $i$ and $p$ implies $a \cong 0$.     
\end{Defn}

Each object $e_i$ is said to be an {\em exceptional object}.
The triangulated subcategory $\langle e_i \rangle$ generated by 
$e_i$ is equivalent to the derived category of a point $D^b(\text{coh}(\text{Spec }k))$.
For an exceptional collection, these subcategories are semi-orthogonal: 
$\langle e_i \rangle \perp \langle e_j \rangle$ for $i > j$.
If the collection is strong and full, then $\bigoplus_{i=1}^n e_i$ is a tilting generator, and there is an equivalence
\[
\mathcal{A} \cong D^b(\text{mod-}\text{End}(\bigoplus_{i=1}^n e_i))
\]
by Theorem~\ref{Rickard}.

\vskip 1pc

We start with the Fano case:

\begin{Thm}[Fano variety \cite{toric}]
Let $X$ be a projective $\mathbf{Q}$-factorial toric variety, let $B$ be a toric $\mathbf{Q}$-divisor
with standard coefficients, and let $\tilde X$ be the smooth Deligne-Mumford stack associated to the pair $(X,B)$.
Assume that $-(K_X+B)$ is ample and that the Picard number $\rho(X) = 1$.  
Then the derived category $D^b(\text{coh}(\tilde X))$ 
is generated by a strong and full exceptional collection consisting of line bundles $(L_1,\dots,L_n)$ on $\tilde X$.
\end{Thm}
 
Let $\pi_X: \tilde X \to X$ be the natural morphism.
Then the direct images $\pi_{X*}L_i$ are reflexive sheaves of rank $1$ on $X$, and they do not satisfy
the vanishings of cohomologies which are necessary for exceptional collections.

A Mori fiber space is a relative version of a Fano variety:
 
\begin{Thm}[Mori fiber space \cite{toric}]\label{Mori}
Let $X$ be a projective $\mathbf{Q}$-factorial toric variety, let $B$ be a toric $\mathbf{Q}$-divisor on $X$ 
with standard coefficients, and let $\tilde X$ be the smooth Deligne-Mumford stack associated to the pair $(X,B)$.
Let $f: X \to Y$ be a surjective toric morphism to another projective $\mathbf{Q}$-factorial toric variety
with connected fibers. 
Assume that $-(K_X+B)$ is $f$-ample and that the relative Picard number $\rho(X/Y) = 1$.  
Then the following hold:

(1) There exists a toric $\mathbf{Q}$-divisor $C$ on $Y$ with standard coefficients, such that, 
if $\tilde Y$ is the smooth Deligne-Mumford stack associated to the pair $(Y,C)$, then $f$ induces 
a smooth morphism $\tilde f: \tilde X \to \tilde Y$.

(2) The functor $\tilde f^*: D^b(\text{coh}(\tilde Y)) \to D^b(\text{coh}(\tilde X))$ is fully faithful.

(3) There exists a sequence of line bundles $(L_1,\dots,L_n)$ on $\tilde X$ 
which is a strong and full {\em relative} exceptional collection in the following sense:
\[
R^p\tilde f_*\mathcal{H}om(L_i,L_j) = \begin{cases} \mathcal{O}_{\tilde Y} \,\, &\text{ if } i = j, p = 0 \\
0 \,\, &\text{ if } \forall i,j, p \ne 0, \\
0 \,\, &\text{ if } i > j, \forall p.
\end{cases}
\]
and there is a semi-orthogonal decomposition:
\[
D^b(\text{coh}(\tilde X)) = \langle \tilde f^*D^b(\text{coh}(\tilde Y)) \otimes L_1, \dots, 
\tilde f^*D^b(\text{coh}(\tilde Y)) \otimes L_n \rangle.
\]
\end{Thm}

Next we consider a divisorial contraction:

\begin{Thm}[divisorial contraction \cite{toric}, \cite{toricII}] 
Let $X$ be a projective $\mathbf{Q}$-factorial toric variety, let $B$ be a toric $\mathbf{Q}$-divisor on $X$ 
with standard coefficients, let $\tilde X$ be the smooth Deligne-Mumford stack associated to the pair $(X,B)$,
and let $\pi_X: \tilde X \to X$ be the natural morphism.
Let $f: X \to Y$ be a toric birational morphism to another projective $\mathbf{Q}$-factorial toric variety
whose exceptional locus is a prime divisor $E$, let $F = f(E) \subset Y$, $g = f \vert_E$, and let $C = f_*B$.
Write $B' = f_*^{-1}C$ and define a $\mathbf{Q}$-divisor $B_E$ on $E$ by an adjunction:
\[
(K_X + B' + E) \vert_E = K_E + B_E.
\]
Then $B_E$ has standard coefficients, and 
$g: E \to F$ for the pair $(E,B_E)$ is a Mori fiber space as in Theorem~\ref{Mori}. 
Let $C_F$ be a toric $\mathbf{Q}$-divisor on $F$ determined in loc. cit. (1).
Let $\tilde Y$, $\tilde E$ and $\tilde F$ be the smooth Deligne-Mumford stacks associated to the pairs 
$(Y,C)$, $(E, B_E)$ and $(F, C_F)$, respectively, and  
let $\pi_Y: \tilde Y \to Y$, $\pi_E: \tilde E \to E$ and $\pi_F: \tilde F \to F$ be the natural morphisms.
There is a smooth morphism $\tilde g: \tilde E \to \tilde F$.
Let $Z = \tilde X \times_Y \tilde Y$, $V = \tilde E \times_X \tilde X$ and $W = \tilde F \times_F \tilde Y$
be the fiber products with natural morphisms $p_1: Z \to \tilde X$ and $p_2: Z \to \tilde Y$,
$q_1: V \to \tilde E$ and $q_2: V \to \tilde X$, and
$r_1: W \to \tilde F$ and $r_2: W \to \tilde Y$, respectively.

\vskip .5pc 

(A) Assume that $K_X+B > f^*(K_Y+C)$.
Then the following hold:

(A-1) The functors $\Phi = p_{1*}p_2^*: D^b(\text{coh}(\tilde Y)) \to D^b(\text{coh}(\tilde X))$ and
$\Psi = q_{2*}q_1^*\tilde g^*: D^b(\text{coh}(\tilde F)) \to D^b(\text{coh}(\tilde X))$ are fully faithful.

(A-2) There exists a sequence of line bundles $(L_1,\dots,L_n)$ on $\tilde X$ 
such that there is a semi-orthogonal decomposition:
\[
D^b(\text{coh}(\tilde X)) = \langle \Psi(D^b(\text{coh}(\tilde F))) \otimes L_1, \dots, 
\Psi(D^b(\text{coh}(\tilde F))) \otimes L_n, \Phi(D^b(\text{coh}(\tilde Y))) \rangle.
\]

\vskip .5pc 

(B) Assume that $K_X+B = f^*(K_Y+C)$.
Then the functor $\Phi = p_{1*}p_2^*: D^b(\text{coh}(\tilde Y)) \to D^b(\text{coh}(\tilde X))$
is an equivalence.

\vskip .5pc 

(C) Assume that $K_X+B < f^*(K_Y+C)$.
Then the following hold:

(C-1) The functors $\Phi = p_{2*}p_1^*: D^b(\text{coh}(\tilde X)) \to D^b(\text{coh}(\tilde Y))$
and $\Psi = r_{2*}r_1^*: D^b(\text{coh}(\tilde F)) \to D^b(\text{coh}(\tilde Y))$ are fully faithful.

(C-2) There exists a sequence of line bundles $(L_1,\dots,L_n)$ on $\tilde Y$ 
such that there is a semi-orthogonal decomposition:
\[
D^b(\text{coh}(\tilde Y)) = \langle \Psi(D^b(\text{coh}(\tilde F))) \otimes L_1, \dots, 
\Psi(D^b(\text{coh}(\tilde F))) \otimes L_n, \Phi(D^b(\text{coh}(\tilde X))) \rangle.
\]
\end{Thm}

In the case (C), the morphism $f: X \to Y$ is often called a {\em divisorial extraction} from $Y$.

Finally we consider a flip:

\begin{Thm}[flip \cite{toric}] 
Let $X$ (resp. $Y$) be a projective $\mathbf{Q}$-factorial toric variety, 
let $B$ (resp. $C$) be a toric $\mathbf{Q}$-divisor on $X$ (resp. $Y$)
with standard coefficients, let $\tilde X$ (resp. $\tilde Y$) 
be the smooth Deligne-Mumford stack associated to the pair $(X,B)$ (resp. $(Y,C)$),
and let $\pi_X: \tilde X \to X$ (resp. $\pi_Y: \tilde Y \to Y$) be the natural morphism.
Let $f: X \to S$ and $f': Y \to S$ be toric birational morphisms to another toric variety
whose exceptional loci $E$ and $E'$ have codimension at least $2$ and such that 
the relative Picard numbers $\rho(X/S) = \rho(Y/S) = 1$.
Let $F = f(E) \subset S$, $g = f \vert_E$.
Assume that $C = f'_*{}^{-1}f_*B$.
Let $E_1,\dots,E_c$ be the toric prime divisors containing $E$ for $c = \text{codim }E$, 
write $B = B' + \sum_{i=1}^c e_iE_i$ where $B'$ does not contain the $E_i$, 
and define a $\mathbf{Q}$-divisor $B_E$ on $E$ by the adjunction 
\[
(K_X+B'+ \sum_{i=1}^c E_i) \vert_E = K_E + B_E.
\]
Then $B_E$ has standard coefficients, and 
$g: E \to F$ for the pair $(E,B_E)$ is a Mori fiber space as in Theorem~\ref{Mori}. 
Let $C_F$ be a toric $\mathbf{Q}$-divisor on $F$ determined in loc. cit. (1).
Let $\tilde E$ and $\tilde F$ be the smooth Deligne-Mumford stacks associated to the pairs 
$(E, B_E)$ and $(F, C_F)$, respectively, and  
let $\pi_E: \tilde E \to E$ and $\pi_F: \tilde F \to F$ be the natural morphisms.
There is a smooth morphism $\tilde g: \tilde E \to \tilde F$.
Let $Z = \tilde X \times_S \tilde Y$ and $V = \tilde E \times_X \tilde X$ be the fiber products 
with natural morphisms $p_1: Z \to \tilde X$ and $p_2: Z \to \tilde Y$, and
$q_1: V \to \tilde E$ and $q_2: V \to \tilde X$, respectively.
Let $h: W \to X$ and $h': W \to Y$ be birational morphisms from a smooth projective variety such that 
$f \circ h = f' \circ h'$.  

\vskip .5pc 

(A) Assume that $h^*(K_X+B) > h'{}^*(K_Y+C)$.
Then the following hold:

(A-1) The functors $\Phi = p_{1*}p_2^*: D^b(\text{coh}(\tilde Y)) \to D^b(\text{coh}(\tilde X))$ and
$\Psi = q_{2*}q_1^*\tilde g^*: D^b(\text{coh}(\tilde F)) \to D^b(\text{coh}(\tilde X))$ are fully faithful.

(A-2) There exists a sequence of line bundles $(L_1,\dots,L_n)$ on $\tilde X$ 
such that there is a semi-orthogonal decomposition:
\[
D^b(\text{coh}(\tilde X)) = \langle \Psi(D^b(\text{coh}(\tilde F))) \otimes L_1, \dots, 
\Psi(D^b(\text{coh}(\tilde F))) \otimes L_n, \Phi(D^b(\text{coh}(\tilde Y))) \rangle.
\]

\vskip .5pc 

(B) Assume that $h^*(K_X+B) =h'{}^*(K_Y+C)$.
Then the functor $\Phi = p_{1*}p_2^*: D^b(\text{coh}(\tilde Y)) \to D^b(\text{coh}(\tilde X))$
is an equivalence.
\end{Thm}

We need an additional result on the change of coefficients:

\begin{Thm}[\cite{log-crep}, \cite{toricIII}]
Let $X$ be a projective $\mathbf{Q}$-factorial toric variety, let $B$ (resp. $C$) be a toric $\mathbf{Q}$-divisor on $X$ 
with standard coefficients, let $\tilde X$ (resp. $\tilde Y$) 
be the smooth Deligne-Mumford stack associated to the pair $(X,B)$ (resp. $(X,C)$),
and let $\pi_X: \tilde X \to X$ (resp. $\pi_Y: \tilde Y \to X$) be the natural morphism.
Assume that $C < B$ and that $B - C$ is supported by a prime divisor $E$, i.e., 
$B = B' + eE$ and $C = B' + e'E$, where $B'$ does not contain $E$.
Define a $\mathbf{Q}$-divisor $B_E$ on $E$ by an adjunction:
\[
(K_X + B' + E) \vert_E = K_E + B_E.
\]
Then $B_E$ has standard coefficients.
Let $\tilde E$ be the smooth Deligne-Mumford stack associated to the pair 
$(E, B_E)$, and  
let $\pi_E: \tilde E \to E$ be the natural morphism.
Let $Z = \tilde X \times_X \tilde Y$ and $V = \tilde E \times_X \tilde X$ be the fiber products 
with natural morphisms $p_1: Z \to \tilde X$ and $p_2: Z \to \tilde Y$, and
$q_1: V \to \tilde E$ and $q_2: V \to \tilde X$.
Then the following hold:

(A-1) The functors $\Phi = p_{1*}p_2^*: D^b(\text{coh}(\tilde Y)) \to D^b(\text{coh}(\tilde X))$ and
$\Psi = q_{2*}q_1^*: D^b(\text{coh}(\tilde E)) \to D^b(\text{coh}(\tilde X))$ are fully faithful.

(A-2) There exists a sequence of line bundles $(L_1,\dots,L_n)$ on $\tilde X$ 
such that there is a semi-orthogonal decomposition:
\[
D^b(\text{coh}(\tilde X)) = \langle \Psi(D^b(\text{coh}(\tilde E))) \otimes L_1, \dots, 
\Psi(D^b(\text{coh}(\tilde E))) \otimes L_n, \Phi(D^b(\text{coh}(\tilde Y))) \rangle.
\]
\end{Thm}

We provide elementary examples of the above theorems:

\begin{Expl}
(1) (divisorial contraction) Let $Y = \mathbf{C}^n/\mathbf{Z}_r$ be a quotient singularity 
defined by an action $(x_1,\dots,x_n) \mapsto (\zeta x_1,\dots,\zeta x_n)$, where 
$\zeta = \exp {2\pi i/r}$.
Let $f: X \to Y$ be a resolution of singularities given by blowing up the origin.
Then we have $K_X = f^*K_Y + (n-r)/rE$, where $E \cong \mathbf{P}^{n-1}$ is the exceptional divisor.
The Deligne-Mumford stack associated to $Y$ is the quotient stack $\tilde Y = [\mathbf{C}^n/\mathbf{Z}_r]$.

If $n > r$, then we have an SOD
\[
D^b(\text{coh}(X)) = \langle \mathcal{O}_E(-n+r), \dots, \mathcal{O}_E(-1), \Phi(D^b(\text{coh}(\tilde Y))) \rangle. 
\]
If $n=r$, then we have an equivalence
$D^b(\text{coh}(X)) \cong D^b(\text{coh}(\tilde Y))$.
If $n < r$, then we have an SOD
\[
D^b(\text{coh}(\tilde Y)) = \langle \mathcal{O}_P(1), \dots, \mathcal{O}_P(r-n), \Phi(D^b(\text{coh}(\tilde X))) \rangle 
\]
where $P = f(E) \in Y$.

\vskip 1pc

(2) (flip) Let $X$ be the total space of a vector bundle $\mathcal{O}_{\mathbf{P}^r}(-1)^{\oplus (s+1)}$
over $\mathbf{P}^r$ for $r,s \ge 1$.
Let $p_1: Z \to X$ be the blowing up along the zero-section $E \cong \mathbf{P}^r$, and let 
$p_2: Z \to Y$ be the blowing down of the exceptional divisor 
$G \cong \mathbf{P}^r \times \mathbf{P}^s$ to the other direction.
Then $Y$ is the total space of a vector bundle $\mathcal{O}_{\mathbf{P}^s}(-1)^{\oplus (r+1)}$.
We have $K_Z = p_1^*K_X + sG = p_2^*K_Y + rG$.

If $r > s$, then we have an SOD
\[
D^b(\text{coh}(X)) = \langle \mathcal{O}_E(-r+s), \dots, \mathcal{O}_E(-1), \Phi(D^b(\text{coh}(Y))) \rangle. 
\]
If $r=s$, then we have an equivalence
$D^b(\text{coh}(X)) \cong D^b(\text{coh}(Y))$.
\end{Expl}

%%%%%%%%%%%%%%%%%%%%%
%%%%%%%%%%%%%%%%%%%%%
%%%%%%%%%%%%%%%%%%%%%
\section{Application to the derived McKay correspondence}\label{GL3}

A derived McKay correspondence is a statement between the derived categories 
of a quotient stack and its resolution of singularities.

A basic example concerns a minimal resolution of a rational double point of a surface
(or a Du Val singularity, or a Kleinian singularity, or a canonical singularity).
Let $Y = \mathbf{C}^2/G$ be a quotient singularity by a finite subgroup
$G \subset SL(2,\mathbf{C})$, and let 
$f: X \to Y$ be a minimal resolution.
Then the theorem states that there is an equivalence (\cite{Kapranov-Vasserot}):
\[
D^b(\text{coh}(X)) \cong D^b(\text{coh}[\mathbf{C}^2/G]).
\]
Since $f$ is {\em crepant}, i.e., $K_X = f^*K_Y$, this is a special case of DK hypothesis.

There are generalizations for other groups in \cite{BKR}, \cite{Kaledin},\cite{Craw-Ishii}, 
\cite{Ishii}, \cite{Ishii-Ueda}.
We treat abelian groups and subgroups in $GL(3,\mathbf{C})$ in this section.

We define minimal and maximal models:

\begin{Defn}
(MI) Let $X$ be a normal variety.
Then a {\em minimal $\mathbf{Q}$-factorial terminalization} or a {\em relative minimal model} 
of $X$ is a projective birational morphism
$f: Y \to X$ from a variety with only $\mathbf{Q}$-factorial terminal singularities 
such that $K_Y$ is $f$-nef.

(MA) Let $(X,B)$ be a KLT pair of a normal variety and an $\mathbf{R}$-divisor.
Then a {\em maximal $\mathbf{Q}$-factorial terminalization} or a {\em relative maximal model} 
of $X$ is a projective birational morphism
$f: Y \to X$ from a variety with only $\mathbf{Q}$-factorial terminal singularities 
which satisfy the following conditions:

\begin{enumerate}
\item $K_Y \le f^*(K_X+B)$.

\item If $f': Y' \to X$ is another projective birational morphism from a variety with only $\mathbf{Q}$-factorial terminal singularities 
satisfying $K_{Y'} \le f'{}^*(K_X+B)$, then the induced birational map $Y \dashrightarrow Y'$ is surjective in codimension $1$. 
\end{enumerate}
\end{Defn}

The existences of minimal and maximal models are guaranteed by \cite{BCHM}.
The uniqueness is not true in general for each model except in dimension $2$. 
All minimal or maximal models are isomorphic in codimension $1$ each other.
Minimal models are $K$-equivalent each other, but maximal models are not in general.
A minimal model of $X$ for a KLT pair $(X,B)$ satisfies the condition (1) of (MA) but not necessarily (2).
A minimal model is {\em crepant}, i.e., $K_Y = f^*K_X$, if $X$ is canonical.

As a corollary of the result in \S \ref{toric}, we obtain a derived McKay correspondence for 
abelian groups:

\begin{Thm}
Let $X = \mathbf{C}^n/G$ be a quotient singularity by a finite abelian subgroup $G \subset GL(n,\mathbf{C})$, and
let $Y \to X$ be a relative minimal model.
Then $Y$ has only abelian quotient singularities.
Let $\tilde Y$ be the associated smooth Deligne-Mumford stack.
Then there are closed toric proper subvarieties $Z_i \subsetneq X$ for 
$i = 1,\dots, m$, allowing repetitions like $Z_i = Z_j$ for $i \ne j$, 
and fully faithful functors $\Phi: D^b(\text{coh}(\tilde Y)) \to D^b(\text{coh}[\mathbf{C}^n/G])$ and
$\Psi_i: D^b(\text{coh}(\tilde Z_i)) \to D^b(\text{coh}[\mathbf{C}^n/G])$, where the $\tilde Z_i$ are 
smooth Deligne-Mumford stacks associated to some minimal models of the $Z_i$, with a semi-orthogonal decomposition
\[
D^b(\text{coh}[\mathbf{C}^n/G]) = \langle \Psi_1(D^b(\text{coh}(\tilde Z_1))), \dots, \Psi_m(D^b(\text{coh}(\tilde Z_m))),
\Phi(D^b(\text{coh}(\tilde Y))) \rangle.
\]
Moreover, if $G \subset SL(n,\mathbf{C})$, then $m = 0$.
\end{Thm}

If we take another minimal model $Y'$ of $X$, we have $Y \sim_K Y'$, hence
$D^b(\text{coh}(\tilde Y)) \cong D^b(\text{coh}(\tilde Y'))$ because DK hypothesis is confirmed for toric
$K$-equivalence. 

By combining the results in \S \ref{toric} with the result of \cite{BKR}, we obtain the following (\cite{GL3C}):

\begin{Thm}
Let $X = \mathbf{C}^3/G$ be a quotient singularity by a finite subgroup $G \subset GL(3,\mathbf{C})$ which is 
not necessarily abelian nor small.
Then there exist a maximal model $Y \to X$ and affine varieties $Z_i$ which are finite over 
closed proper subvarieties of $X$ for 
$i = 1,\dots, m$, allowing repetitions like $Z_i = Z_j$ for $i \ne j$, 
and fully faithful functors $\Phi: D^b(\text{coh}(\tilde Y)) \to D^b(\text{coh}[\mathbf{C}^n/G])$ and
$\Psi_i: D^b(\text{coh}(\tilde Z_i)) \to D^b(\text{coh}[\mathbf{C}^n/G])$, where the $\tilde Z_i$ are 
minimal models of the $Z_i$, with a semi-orthogonal decomposition
\[
D^b(\text{coh}[\mathbf{C}^3/G]) = \langle \Psi_1(D^b(\text{coh}(\tilde Z_1))), \dots, \Psi_m(D^b(\text{coh}(\tilde Z_m))), 
\Phi(D^b(\text{coh}(\tilde Y))) \rangle.
\]
\end{Thm}

There are choices of maximal models which are not $K$-equivalent.
Therefore we ask whether the theorem is true for other maximal models.

%%%%%%%%%%%%%%%%%%%%%%%%%%%%%%%%%%%%%%%%%%

Graduate School of Mathematical Sciences, University of Tokyo,
Komaba, Meguro, Tokyo, 153-8914, Japan 

kawamata@ms.u-tokyo.ac.jp

\end{document}